\documentclass[preprint,12pt]{elsarticle}

\usepackage{amsfonts,amsthm,amsmath}
\usepackage[dvips]{epsfig}
\usepackage[dvips]{graphicx}

\theoremstyle{plain}
\numberwithin{equation}{section}

\newtheorem{lemma}[equation]{Lemma}
\newtheorem{proposition}[equation]{Proposition}

\newtheorem{corollary}[equation]{Corollary}

\newtheorem{conjecture}[equation]{Conjecture}
\newtheorem{openprob}[equation]{Open problem}
\newtheorem{definition}[equation]{Definition}

\begin{document}

\begin{frontmatter}

\title{Rich square-free words}

\author{Jetro Vesti}
\ead{jejove@utu.fi}

\address{University of Turku, Department of Mathematics and Statistics, 20014 Turku, Finland}

\begin{abstract}
A word $w$ is \emph{rich} if it has $|w|+1$ many distinct palindromic factors, including the empty word.
A word is \emph{square-free} if it does not have a factor $uu$, where $u$ is a non-empty word.

Pelantov\'a and Starosta (Discrete Math. 313 (2013)) proved that every infinite rich word contains a square.
We will give another proof for that result.
Pelantov\'a and Starosta denoted by $r(n)$ the length of a longest rich square-free word on an alphabet of size $n$.
The exact value of $r(n)$ was left as an open question. We will give an upper and a lower bound for $r(n)$,
and make a conjecture that our lower bound is exact.

We will also generalize the notion of repetition threshold for a limited class of infinite words.
The repetition thresholds for episturmian and rich words are left as an open question.
\end{abstract}


\journal{Theoretical Computer Science}

\end{frontmatter}

%%%%%%%%%%%%%%%%%%%%%%%%%%%%%%%%%%%%%%%%%%%
%%%%%%%%%%%%%%%%%%%%%%%%%%%%%%%%%%%%%%%%%%%
%%%%%%%%%%%%%%%%%%%%%%%%%%%%%%%%%%%%%%%%%%%

\section{Introduction}

In recent years, rich words and palindromes have been studied extensively in combinatorics on words.
A word is a \emph{palindrome} if it is equal to its reversal.
In \cite{djp}, the authors proved that every word $w$ has at most $|w|+1$ many distinct palindromic factors, including the empty word.
The class of words which achieve this limit was introduced in \cite{bhnr} with the term \emph{full} words.
When the authors of \cite{gjwz} studied these words thoroughly they called them \emph{rich} (in palindromes).
Rich words have been studied in various papers, for example in \cite{afmp}, \cite{bdgz1}, \cite{bdgz}, \cite{lgz}, \cite{rr} and \cite{v}.

The \emph{defect} of a finite word $w$, denoted $D(w)$, is defined as $D(w)=|w|+1-|\textrm{Pal}(w)|$, where $\textrm{Pal}(w)$ is
the set of palindromic factors in $w$. The \emph{defect} of an infinite word $w$ is defined as $D(w)=\textrm{sup}\{D(u)|\ u\ \textrm{is a factor of}\ w\}$.
In other words, the defect tells how many palindromes the word is lacking.
Rich words are exactly those whose defect is equal to 0.

The authors of \cite{ps} proved, in Theorem 4 of the article, that every recurrent word with finite $\Theta$-defect contains infinitely many overlapping factors.
An \emph{overlapping} word is a word of form $uuv$, where $v$ is a non-empty prefix of $u$.
A word is a $\Theta$-\emph{palindrome} if it is a fixed point of an involutive antimorphism $\Theta$.
The reversal mapping $R$ is an involutive antimorphism, which means that if $\Theta=R$ then $\Theta$-defect is equal to the defect.
This means Theorem 4 in \cite{ps} holds also for normal defect and normal palindromes.
In this article we will restrict ourselves to the case where $\Theta$ is the reversal mapping.

Since every rich word has a finite defect and every overlapping factor $uuv$ has a square $uu$,
a corollary of Theorem 4 in \cite{ps} is that every recurrent rich word contains a square.
This was noted in \cite{ps} as Remark 6, where the word \emph{recurrent} was replaced with \emph{infinite}.
This can be done, since every infinite rich word $x$ has a recurrent point $y$ in the shift orbit closure of $x$ (see e.g. Section 4 of \cite{q}).
We know $y$ has a square, which means $x$ has a square.
In Corollary \ref{cor1} of our article, we will give another proof of the result in Remark 6 of \cite{ps}.

In Remark 6 of \cite{ps} there was also noted that since every rich square-free word is finite, we can look for a longest one.
The length of a longest such word, on an alphabet of size $n$, was denoted by $r(n)$.
An explicit formula for $r(n)$ was left as an open question.

In Section \ref{lower} we will construct recursively a sequence of rich square-free words, the lengths of which give us a lower bound for $r(n)$.
We will also make a conjecture that $r(n)$ can be achieved using these words.
In Section \ref{upper} we will prove an upper bound for $r(n)$.

\subsection{Repetition threshold}

Square-free words are a special case of unavoidable repetitions of words, which has been a central topic in combinatorics on words since Thue (see \cite{t1} and \cite{t2}).
The \emph{repetition threshold}, on an alphabet of size $n$, is the smallest number $r$ such that there exists an infinite word which avoids greater than $r$-powers.
This number is denoted by $RT(n)$ and it was first studied in \cite{d}, where Dejean gave her famous conjecture.
This conjecture has now been proven, in many parts and by several authors (see \cite{r} and \cite{cr}).

The repetition threshold can be studied also for a limited class of infinite words.
In \cite{mp}, it was proven that the infinite Fibonacci word does not contain a power with exponent greater than $2+\varphi$, where $\varphi$ is the golden ratio $\frac{\sqrt{5}+1}{2}$, but every smaller fractional power is contained.
In \cite{cd}, the authors proved that among \emph{Sturmian} words, the Fibonacci word is optimal with respect to this property.
Sturmian words are equal to \emph{episturmian} words when $n=2$ (see \cite{djp}).
This means the \emph{episturmian repetition threshold} for $n=2$ is $2+\varphi$, denoted $ERT(2)=2+\varphi$.
From \cite{gj}, we get that the $n$-bonacci word is episturmian and it has critical exponent $2+1/(\varphi_n-1)$, where $\varphi_n$ is the generalized golden ratio.
This means $ERT(n)\leq 2+1/(\varphi_n-1)$. Notice, from \cite{hps} we get that $\varphi_n$ converges to $2$.

In the same way, we define the \emph{rich repetition threshold} $RRT(n)$. From \cite{ps} we get that $RRT(n)\geq 2$.
Since episturmian words are rich (see \cite{djp}), we also know $RRT(n)\leq 2+1/(\varphi_n-1)$ and $ERT(n)\geq 2$. This means $2\leq RRT(n),ERT(n)\leq 2+1/(\varphi_n-1)$.
The exact values of $ERT(n)$ and $RRT(n)$ are left as an open problem.

\begin{openprob}\label{open}
Determine the repetition threshold for episturmian words and for rich words, on an alphabet of size $n$.
\end{openprob}

\subsection{Preliminaries}

An \emph{alphabet} $A$ is a non-empty finite set of symbols, called \emph{letters}.
A \emph{word} is a finite sequence of letters from $A$.
The \emph{empty} word $\epsilon$ is the empty sequence.
The set $A^*$ of all finite words over $A$ is a \emph{free monoid} under the operation of concatenation.
The set $\textrm{Alph}(w)$ is the set of all letters that occur in $w$.
If $|\textrm{Alph}(w)|=n$ then we say that $w$ is \emph{$n$-ary}.

An \emph{infinite word} is a sequence indexed by $\mathbb{N}$ with values in $A$.
We denote the set of all infinite words over $A$ by $A^\omega$ and define $A^\infty=A^*\cup A^\omega$.

The \emph{length} of a word $w=a_1 a_2 \ldots a_n$, with each $a_i\in A$, is denoted by $|w|=n$.
The empty word $\epsilon$ is the unique word of length $0$. By $|w|_a$ we denote the number of occurrences of a letter $a$ in $w$. 

A word $x$ is a \emph{factor} of a word $w\in A^\infty$, denoted $x\in w$, if $w=uxv$ for some $u\in A^*$,$v\in A^\infty$.
If $x$ is not a factor of $w$, we denote $x\notin w$.
If $u=\epsilon$ (resp. $v=\epsilon$) then we say that $x$ is a \emph{prefix} (resp. \emph{suffix}) of $w$.
If $w=uv\in A^*$ is a word, we use the notation $u^{-1}w=v$ or $wv^{-1}=u$ to mean the removal of a prefix or a suffix of $w$.
We say that a prefix or a suffix of $w$ is \emph{proper} if it is not the whole $w$.

A factor $x$ of a word $w$ is said to be \emph{unioccurrent} in $w$ if $x$ has exactly one occurrence in $w$.
Two occurrences of factor $x$ are said to be \emph{consecutive} if there is no occurrence of $x$ between them.
A factor of $w$ having exactly two occurrences of a non-empty factor $u$, one as a prefix and the other as a suffix,
is called a \emph{complete return} to $u$ in $w$.

The \emph{reversal} of $w=a_1 a_2\ldots a_n$ is defined as $\widetilde{w}=a_n\ldots a_2 a_1$.
A word $w$ is called a \emph{palindrome} if $w=\widetilde{w}$. The empty word $\epsilon$ is assumed to be a palindrome.

Other basic definitions and notation in combinatorics on words can be found from Lothaire's books \cite{l1} and \cite{l2}.

\begin{proposition}\label{p0}(\cite{djp}, Prop. 2)
A word $w$ has at most $|w|+1$ distinct palindromic factors, including the empty word.
\end{proposition}

\begin{definition}
A word $w$ is \emph{rich} if it has exactly $|w|+1$ distinct palindromic factors, including the empty word.
An infinite word is \emph{rich} if all of its factors are rich.
\end{definition}

\begin{proposition}\label{p3}(\cite{gjwz}, Thm. 2.14)
A finite or infinite word $w$ is rich if and only if all complete returns to any palindromic factor in $w$ are themselves palindromes.
\end{proposition}

Let $w=vu$ be a word and $u$ its longest palindromic suffix. The \emph{palindromic closure} of $w$ is defined as $w^{(+)}=vu\widetilde{v}$.
If $u$ is the longest \emph{proper} palindromic suffix of $w$, called ${\rm lpps}$, we define the \emph{proper palindromic closure} of $w$ the same way as $w^{(++)}=vu\tilde{v}$.
We refer to the longest proper palindromic prefix of $w$ as ${\rm lppp}$ and define
the \emph{proper palindromic prefix closure} of $w$ as $^{(++)}w=\widetilde{\widetilde{w}^{(++)}}$.

\begin{proposition}\label{p4}(\cite{gjwz}, Prop. 2.6)
Palindromic closure preserves richness. 
\end{proposition}

\begin{proposition}\label{p5}(\cite{gjwz}, Prop. 2.8)
Proper palindromic (prefix) closure preserves richness.
\end{proposition}

%%%%%%%%%%%%%%%%%%%%%%%%%%%%%%%%%%%%%%%%%%%
%%%%%%%%%%%%%%%%%%%%%%%%%%%%%%%%%%%%%%%%%%%
%%%%%%%%%%%%%%%%%%%%%%%%%%%%%%%%%%%%%%%%%%%

\section{The length of a longest rich square-free word}

A word of form $uu$, where $u\neq\epsilon$, is called a \emph{square} and a word $w$ which does not have a square as a factor is called \emph{square-free}.
For example 1212 is a square and 01210 is square-free.

In \cite{ps}, Theorem 4 and Remark 6, it was proved that every infinite rich word contains a square.
This means that every rich square-free word is of finite length.
The length of a longest such word, on an alphabet of size $n$, is denoted with $r(n)$.
The explicit formula for $r(n)$ was left as an open problem in \cite{ps}.

The first seven exact values of $r(n)$ are $r(1)=1,r(2)=3,r(3)=7,r(4)=15,r(5)=33,r(6)=67$ and $r(7)=145$. These can be found from https://oeis.org/A269560.
The longest rich square-free word on a given alphabet is not unique.
Here are all the longest non-isomorphic ones, up to permutating the letters and taking the reversal, for $n=1,\ldots,7$:

\noindent $w_{1,1} = 1$

\noindent $w_{2,1} = 121$

\noindent $w_{3,1} = 2131213$

\noindent $w_{3,2} = 1213121$

\noindent $w_{4,1} = 131214121312141$

\noindent $w_{4,2} = 123121412131214$

\noindent $w_{4,3} = 213121343121312$

\noindent $w_{4,4} = 121312141213121$

\noindent $w_{5,1} = 421242131213531213124213121353135$

\noindent $w_{5,2} = 131242131213531213124213121353135$

\noindent $w_{6,1} = 1513121315131214121312141614121312141213151312141213121416141214161$

\noindent $w_{6,2} = 1214121315131214121312141614121312141213151312141213121416141214161$

\noindent $w_{6,3} = 4212421312135312131242131213531356531353121312421312135312131242124$

\noindent $w_{6,4} = 1312421312135312131242131213531356531353121312421312135312131242124$

\noindent $w_{6,5} = 5313531213124213121353121312421316131242131213531213124213121353135$

\noindent $w_{6,6} = 1312421312135312131242131213531356531353121312421312135312131242131$

\noindent $w_{7,1} = 242131213531213124213161312421312135312131242131213531357531353121312$

\noindent $4213121353121312421316131242131213531213124213121353135753135312135313575357$

\noindent $w_{7,2} = 242131213531213124213161312421312135312131242131213531357531353121312$

\noindent $4213121353121312421316131242131213531213124213121353135753135312135313575313$

\noindent $w_{7,3} = 242131213531213124212464212421312135312131242131213531357531353121312$

\noindent $4213121353121312421246421242131213531213124213121353135753135312135313575357$

\noindent $w_{7,4} = 242131213531213124212464212421312135312131242131213531357531353121312$

\noindent $4213121353121312421246421242131213531213124213121353135753135312135313575313$

We can see that
$$w_{2,1}=w_{1,1} 2 w_{1,1},\ w_{3,2}=w_{2,1} 3 w_{2,1},\ w_{4,3}=w_{3,1} 4 \widetilde{w_{3,1}},\ w_{4,4}=w_{3,2} 4 w_{3,2},$$
$$w_{6,3} = w_{5,1} 6 \widetilde{w_{5,1}},\ w_{6,4} = w_{5,2} 6 \widetilde{w_{5,1}},\ w_{6,5} = \widetilde{w_{5,2}} 6 w_{5,2}\ \textrm{and}\ w_{6,6} = w_{5,2} 6 \widetilde{w_{5,2}}.$$
Generally, we can construct rich square-free words by using a basic recursion
$$b_n = b a \widetilde{b},$$
where $b$ is a longest rich square-free word over an $(n-1)$-ary alphabet $A$ and $a\notin A$ is a new letter.
It is very easy to see that $b_n$ is rich and square-free. This gives us a recursive lower bound for $r(n)$: $r(n)\geq 2 r(n-1)+1$, for all $n\geq2$.
We will use this inequality excessively later in Section \ref{upper}, when we prove an upper bound for $r(n)$.
The closed-form solution for the recursion $r(1)=1,r(n)\geq 2 r(n-1)+1$ is $r(n)\geq 2^n-1$.

The case $n=5$ reveals that the basic recursion $b_n = b a \widetilde{b}$ is not always optimal,
since neither $w_{5,1}$ nor $w_{5,2}$ is of that form: $|w_{5,1}|=r(5)=33>31=2\cdot r(4)+1$.

We can also see that
$$w_{3,1}=2 w_{1,1} 3 w_{1,1} 2 w_{1,1} 3, \ w_{4,1}=1 3 w_{2,1} 4 w_{2,1} 3 w_{2,1} 4 1, \ w_{4,2}=213 w_{2,1} 4 w_{2,1} 3 w_{2,1} 4,$$
$$w_{5,1} = 42124 w_{3,1} 5 \widetilde{w_{3,1}} 4 w_{3,1} 53135, \ w_{5,2} = 13124 w_{3,1} 5 \widetilde{w_{3,1}} 4 w_{3,1} 53135,$$
$$w_{6,1} = 1513121315 w_{4,1} 6 \widetilde{w_{4,1}} 5 w_{4,1} 6141214161,\
w_{6,2} = 1214121315 w_{4,1} 6 \widetilde{w_{4,1}} 5 w_{4,1} 6141214161,$$
$$w_{7,1} = u_{1,2} 6 w_{5,2} 7 \widetilde{w_{5,2}} 6 w_{5,2} 7 v_{1,3},\
w_{7,2} = u_{1,2} 6 w_{5,2} 7 \widetilde{w_{5,2}} 6 w_{5,2} 7 v_{2,4},$$
$$w_{7,3} = u_{3,4} 6 w_{5,1} 7 \widetilde{w_{5,1}} 6 w_{5,1} 7 v_{1,3},\
w_{7,4} = u_{3,4} 6 w_{5,1} 7 \widetilde{w_{5,1}} 6 w_{5,1} 7 v_{2,4},$$
$$\textrm{where}\  u_{1,2}=2421312135312131242131, u_{3,4}=2421312135312131242124,$$
$$v_{1,3}=53135312135313575357 \ \textrm{and}\  v_{2,4}=53135312135313575313.$$
This gives us a hint how to get, in some cases, a better recursion than the basic recursion.
We will define this recursion explicitly in the next subsection.

\subsection{A lower bound}\label{lower}

In this subsection, we will prove another lower bound for $r(n)$.
We will use an alphabet $\{A_0,A_1,A_2,A_3,B_3,A_4,B_4,A_5,B_5,\ldots\}$.
The following construction of rich square-free words $w_n$ is recursive. The first six words are
$$w_1=A_1, w_2=A_0 A_2 A_0, w_3 = v_3 A_3 w_1 B_3 w_1 A_3 w_1 B_3 u_3, w_4 = v_4 A_4 w_2 B_4 w_2 A_4 w_2 B_4 u_4,$$
$$w_5 = v_5 A_5 w_3 B_5 w_3 A_5 w_3 B_5 u_5, w_6 = v_6 A_6 w_4 B_6 w_4 A_6 w_4 B_6 u_6,$$
where $v_3,u_3=\epsilon$, $v_4,u_4=A_0$, $v_5=A_5 A_3 A_1 A_3$, $u_5=B_3 A_1 A_3 A_1$, $v_6=A_0 A_6 A_0 A_4 A_0 A_2 A_0 A_4 A_0$ and $u_6=A_0 B_4 A_0 A_2 A_0 A_4 A_0 A_2 A_0$.
Notice that $w_6$ is isomorphic ($\cong$) to $w_{6,2}$, $w_5\cong w_{5,2}$, $w_4\cong w_{4,1}$ and $w_3\cong w_{3,1}$.
For $n\geq7$, we define $$w_n = v_n A_n w_{n-2} B_n \widetilde{w_{n-2}} A_n w_{n-2} B_n u_n,$$
where $v_n = (P_n c_n)^{-1} \widetilde{v_{n-4}} A_{n-2} \widetilde{v_{n-2}} A_n v_{n-2} A_{n-2} v_{n-4} A_{n-4} w_{n-6} B_{n-4} \widetilde{w_{n-6}} A_{n-4} \widetilde{v_{n-4}} A_{n-2}\widetilde{v_{n-2}}$
and $u_n = \widetilde{u_{n-2}} B_{n-2} \widetilde{w_{n-4}} A_{n-2} w_{n-4} B_{n-2} \widetilde{u_{n-4}} B_{n-4} \widetilde{w_{n-6}} (d_n \widetilde{P_n})^{-1}$,
where $P_n$ is the largest common prefix of $w_{n-6}$ and $\widetilde{v_{n-4}}$, $c_n$ is the first letter of $(P_n)^{-1}\widetilde{v_{n-4}}A_{n-2}$ and $d_n$ is the first letter of $(P_n)^{-1}w_{n-6}B_{n-4}$.

We can see that $\textrm{Alph}(w_{2k})=\{A_0,A_2,A_4,B_4,A_6,B_6,\ldots,A_{2k},B_{2k}\}$ and $\textrm{Alph}(w_{2k+1})=\{A_1,A_3,B_3,A_5,B_5,\ldots,A_{2k+1},B_{2k+1}\}$.
This means we really have $\textrm{Alph}(w_n)=n$.
We also have $c_n\neq d_n$, since $A_{n-2}\notin w_{n-6}$ and $B_{n-4}\notin \widetilde{v_{n-4}}$.

Before we prove that $w_n$ is rich and square-free, we will make some notation in order to make the proof look simpler.
We mark that $E_n = A_n w_{n-2} B_n \widetilde{w_{n-2}} A_n w_{n-2} B_n$,
$F_n = (P_n c_n)^{-1} \widetilde{v_{n-4}} A_{n-2} \widetilde{v_{n-2}}$,
$G_n = \widetilde{w_{n-6}} A_{n-4} \widetilde{v_{n-4}} A_{n-2} \widetilde{v_{n-2}}$
and $H_n = \widetilde{P_n} A_{n-4} w_{n-6} B_{n-4} G_n$.
Now $w_n = v_n E_n u_n$, $v_n = F_n A_n \widetilde{G_n} B_{n-4} G_n$ and $w_{n-2}=\widetilde{H_n} d_n \widetilde{u_n}$.
We can also see that $H_n$ is a suffix of $v_n$ and $F_n$ is a suffix of $G_n$.

\begin{proposition}\label{p2.1}
The word $w_n$ is square-free for all $n\geq1$.
\end{proposition}
\begin{proof}
We prove the claim by induction. It is easy to check that $w_n$ is square-free when $1\leq n\leq6$.
Suppose $w_n$ is square-free for all $n<k$, where $k\geq7$. Now we need to prove that $w_k$ is square-free.

The word $A_k w_{k-2} B_k \widetilde{w_{k-2}} A_k w_{k-2} B_k u_k$ is square-free because $w_{k-2}$ is square-free, $A_k,B_k\notin w_{k-2}$ and $\widetilde{u_k}$ is a proper suffix of $w_{k-2}$.
The words $G_k$ and $F_k$ are suffixes of $\widetilde{w_{k-2}}$ and $A_k,B_{k-4}\notin G_k,F_k$, which means that $v_k = F_k A_k \widetilde{G_k} B_{k-4} G_k$ is square-free.

Now, the only way $w_k = F_k A_k \widetilde{G_k} B_{k-4} G_k A_k w_{k-2} B_k \widetilde{w_{k-2}} A_k w_{k-2} B_k u_k$ can have a square is either
1) $x A_k w_{k-2} B_k y x A_k w_{k-2} B_k y$, where $x$ is a suffix of both $v_k$ and $\widetilde{w_{k-2}}$, and $y$ is a prefix of both $u_k$ and $\widetilde{w_{k-2}}$,
or 2) $x A_k y x A_k y$, where $x$ is a suffix of both $F_k$ and $\widetilde{G_k} B_{k-4} G_k$, and  $y$ is a prefix of both $\widetilde{w_{k-2}}$ and $\widetilde{G_k} B_{k-4} G_k$.

1) Case $x A_k w_{k-2} B_k y x A_k w_{k-2} B_k y$.
Now $yx = \widetilde{w_{k-2}}=u_k d_k H_k$.
Because $y$ is a prefix of $u_k$ and $x$ is suffix of $v_k$, we have that $d_k H_k$ is a suffix of $v_k$.
We also know that $c_k H_k$ is always a suffix of $v_k$.
This is a contradiction since $c_k \neq d_k$.

2) Case $x A_k y x A_k y$.
Now $y$ is a prefix of $\widetilde{w_{k-2}}$, which means that $x$ has to have a suffix $P_k^{-1} \widetilde{v_{k-4}} A_{k-2} \widetilde{v_{k-2}}$.
This is a contradiction, since $x$ is also a suffix of $(P_k c_k)^{-1} \widetilde{v_{k-4}} A_{k-2} \widetilde{v_{k-2}}$.
\end{proof}

\begin{proposition}\label{p2.2}
The word $w_n$ is rich for all $n\geq1$.
\end{proposition}
\begin{proof}
We prove the claim by induction. It is easy to check that $w_n$ is rich when $1\leq n\leq6$.
Suppose $w_n$ is rich for all $n<k$, where $k\geq7$. Now we need to prove that $w_k$ is rich.

Since $w_{k-2}$ is rich and $A_k,B_k\notin w_{k-2}$, we get that $A_k w_{k-2} B_k$ is rich.
Proposition \ref{p4} gives now that $A_k w_{k-2} B_k \widetilde{w_{k-2}} A_k$ is rich.
The ${\rm lpps}$ of $A_k w_{k-2} B_k \widetilde{w_{k-2}} A_k$ is $A_k$, which means $A_k w_{k-2} B_k \widetilde{w_{k-2}} A_k w_{k-2} B_k \widetilde{w_{k-2}} A_k$ is rich by Proposition \ref{p5}.
The word $u_k$ is a prefix of $\widetilde{w_{k-2}}$, so the factor $A_k w_{k-2} B_k \widetilde{w_{k-2}} A_k w_{k-2} B_k u_k$ is also rich.

The ${\rm lppp}$ of $A_k w_{k-2} B_k \widetilde{w_{k-2}} A_k w_{k-2} B_k u_k$ is $A_k w_{k-2} B_k \widetilde{w_{k-2}} A_k$, which means that also
the proper palindromic prefix closure $\widetilde{u_{k}} B_k \widetilde{w_{k-2}} A_k w_{k-2} B_k \widetilde{w_{k-2}} A_k w_{k-2} B_k u_k$ is rich.
The word $H_k$ is a suffix of $\widetilde{w_{k-2}}$, which means $H_k A_k w_{k-2} B_k \widetilde{w_{k-2}} A_k w_{k-2} B_k u_k = H_k E_k u_k$ is also rich.

The word $c_k H_k E_k u_k$ has a palindromic prefix $PP = c_k \widetilde{P_k} A_{k-4} w_{k-6} B_{k-4} \widetilde{w_{k-6}} A_{k-4} P_k c_k$.
The following paragraph proves that it is unioccurrent in $c_k H_k E_k u_k$.

The letter $B_{k-4}$ occurs only once in $c_k H_k$, in the middle of our palindromic prefix $PP$. This occurrence of $B_{k-4}$ is preceded by $c_k \widetilde{P_k} A_{k-4} w_{k-6}$ and succeeded by $\widetilde{w_{k-6}} A_{k-4} P_k c_k$.
The last occurrence of $B_{k-4}$ in $E_k u_k$ is succeeded by $\widetilde{w_{k-6}} (d_k \widetilde{P_k})^{-1}$ and nothing more.
Since the word $\widetilde{w_{k-6}} (d_k \widetilde{P_k})^{-1}$ is clearly a proper prefix of $\widetilde{w_{k-6}} A_{k-4} P_k c_k$, this last occurrence of $B_{k-4}$ in $E_k u_k$ cannot occur in a factor $PP$.
All other occurrences of $B_{k-4}$ in $E_k u_k$ are preceded by $B_{k-4} \widetilde{w_{k-6}} A_{k-4} w_{k-6}$ or succeeded by $\widetilde{w_{k-6}} A_{k-4} w_{k-6} B_{k-4}$.
The word $B_{k-4} \widetilde{w_{k-6}} A_{k-4} w_{k-6}$ has a suffix $d_k \widetilde{P_k} A_{k-4} w_{k-6}$, which means that it cannot have a suffix $c_k \widetilde{P_k} A_{k-4} w_{k-6}$ because $c_k \neq d_k$.
These mean that no $B_{k-4}$ in $c_k H_k E_k u_k$ can occur in a factor $PP$, except the first one.

Since $PP$ is unioccurrent palindromic prefix in $c_k H_k E_k u_k$, we get that $c_k H_k E_k u_k$ is rich and $PP$ is the ${\rm lppp}$ of $c_k H_k E_k u_k$.
Now, all we need to do is to take the proper palindromic prefix closure of $c_k H_k E_k u_k$, which is rich by Proposition \ref{p5}.
It has a suffix $w_k$, which concludes the proof:
$$^{(++)}(c_k H_k E_k u_k) = \widetilde{u_k} B_k \widetilde{w_{k-2}} A_k w_{k-2} B_k \widetilde{w_{k-2}} A_k \widetilde{G_k} B_{k-4} G_k E_k u_k$$
$$\stackrel{*}{=} X F_k A_k \widetilde{G_k} B_{k-4} G_k E_k u_k = X v_k E_k u_k = X w_k\ (^* F_k\ \textrm{is a suffix of}\ \widetilde{w_{k-2}}).$$
\end{proof}

Now we know that $w_n$ is rich and square-free, which means $r(n)\geq |w_n|$ for all $n\geq1$.
We can compute that $|w_7|=145$, $|w_8|=291$, $|w_9|=629$ and $|w_{10}|=1255$.
Notice that $w_7=w_{7,4}$, which means our lower bound is exact when $n=7$.
The cases $r(8)$ and $r(9)$ are too large to compute the exact value.
However, by creating a partial tree of rich square-free words for $n=8$ and $9$, by leaving some branches out of it,
the longest words we could find were of length 291 and 629, respectively.
These are exactly the lengths of $|w_8|$ and $|w_9|$.
Notice that $|w_8|=291=2\cdot145+1=2|w_7|+1$, which means the basic recursion $b_n$ is as good as our recursion $w_n$ when $n=8$.
Notice also that $|w_9|=629>583=2\cdot 291+1=2|w_8|+1$ and $|w_{10}|=1255<1259=2\cdot 629+1=2|w_9|+1$,
which mean $w_n$ is better than $b_n$ when $n=9$ and $b_n$ is better than $w_n$ when $n=10$.

The previous paragraph suggests that it is reasonable to make the following conjecture.

\begin{conjecture}\label{con}
$r(n)=\textrm{max}\{|w_n|,2\cdot|w_{n-1}|+1\}$ for all $n\geq1$.
\end{conjecture}

The recursion for the length of $w_n$ might be too complex to be solved in a closed-form, but we want to get at least an estimate for it.
Let us first estimate the length of $v_n$, which will be used in Proposition \ref{pro11}.

\begin{lemma}\label{lem11}
$|v_n|\geq 3|v_{n-2}|+2|w_{n-6}|+2|v_{n-4}|+6$, for $n\geq 7$.
\end{lemma}
\begin{proof}
$$|v_n|= |(P_n c_n)^{-1} \widetilde{v_{n-4}} A_{n-2}\widetilde{v_{n-2}} A_n v_{n-2} A_{n-2} v_{n-4} A_{n-4} w_{n-6} B_{n-4} \widetilde{w_{n-6}} A_{n-4} \widetilde{v_{n-4}} A_{n-2} \widetilde{v_{n-2}}|$$
$$\geq |\widetilde{v_{n-2}} A_n v_{n-2} A_{n-2} v_{n-4} A_{n-4} w_{n-6} B_{n-4} \widetilde{w_{n-6}} A_{n-4} \widetilde{v_{n-4}} A_{n-2} \widetilde{v_{n-2}}|$$
$$\geq 3|v_{n-2}|+2|w_{n-6}|+2|v_{n-4}|+6,$$
where $|(P_n c_n)^{-1} \widetilde{v_{n-4}}A_{n-2}|\geq 0$, since $c_n$ is a letter and $P_n$ is a prefix of $\widetilde{v_{n-4}}$.
\end{proof}

\begin{proposition}\label{pro11}
$r(n)\geq|w_n|> 2,008^n$ for $n\geq 5$.
\end{proposition}
\begin{proof}
From our recursion of $w_n$, we get that for $n\geq 11$:
$$|w_n|=|v_n A_n w_{n-2} B_n \widetilde{w_{n-2}} A_n w_{n-2} B_n u_n| = 3|w_{n-2}|+|v_n|+|u_n|+4$$
$$ = 3|w_{n-2}|+|(P_n c_n)^{-1} \widetilde{v_{n-4}} A_{n-2} \widetilde{v_{n-2}} A_n v_{n-2} A_{n-2} v_{n-4} A_{n-4} w_{n-6} B_{n-4} \widetilde{w_{n-6}} A_{n-4} \widetilde{v_{n-4}} A_{n-2} \widetilde{v_{n-2}}|$$
$$+|\widetilde{u_{n-2}} B_{n-2} \widetilde{w_{n-4}} A_{n-2} w_{n-4} B_{n-2} \widetilde{u_{n-4}} B_{n-4} \widetilde{w_{n-6}} (d_n \widetilde{P_n})^{-1}|+4$$
$$= 3|w_{n-2}|+ |(P_n c_n)^{-1} \widetilde{v_{n-4}} A_{n-2} \widetilde{v_{n-2}} A_n v_{n-2} A_{n-2} v_{n-4}|-|d_n \widetilde{P_n}|+4$$
$$+|\widetilde{u_{n-2}} B_{n-2} \widetilde{w_{n-4}} A_{n-2} w_{n-4} B_{n-2} \widetilde{u_{n-4}} B_{n-4} \widetilde{w_{n-6}}|+|A_{n-4} w_{n-6} B_{n-4} \widetilde{w_{n-6}} A_{n-4} \widetilde{v_{n-4}} A_{n-2} \widetilde{v_{n-2}}| $$
$$= 4|w_{n-2}|+ |(P_n c_n)^{-1} \widetilde{v_{n-4}} A_{n-2} \widetilde{v_{n-2}} A_n v_{n-2} A_{n-2} v_{n-4}|-|d_n \widetilde{P_n}|+4$$
$$= 4|w_{n-2}|+2(|\widetilde{v_{n-4}}|-|P_n|)+2|v_{n-2}|+|A_{n-2} A_n A_{n-2}|-|d_n|-|c_n|+4$$
$$\geq 4|w_{n-2}|+2|v_{n-2}|+5 \geq 4|w_{n-2}|+2(3|v_{n-4}|+2|w_{n-8}|+2|v_{n-6}|+6)+5$$
$$\geq 4|w_{n-2}|+2(3(3|v_{n-6}|+2|w_{n-10}|+2|v_{n-8}|+6)+2|w_{n-8}|+2|v_{n-6}|+6)+5$$
$$> 4|w_{n-2}|+4|w_{n-8}|+12|w_{n-10}|.$$
From our recursion of $w_n$ we also know that
$|w_{10}|=1255>1164=4|w_8|$,
$|w_9|=629> 580=4|w_7|$,
$|w_8|=291> 268=4|w_6|$ and
$|w_7|=145> 132=4|w_5|$.

Now, for $n\geq 15$ we have
$$|w_n| > 4|w_{n-2}|+4|w_{n-8}|+12|w_{n-10}|> 4(4|w_{n-4}|+4|w_{n-10}|)+4|w_{n-8}|+12|w_{n-10}|$$
$$= 16|w_{n-4}|+4|w_{n-8}|+28|w_{n-10}|> 16\cdot4\cdot4\cdot4|w_{n-10}|+4\cdot4|w_{n-10}|+28|w_{n-10}|$$
$$= 1068|w_{n-10}|>2,008^{10}|w_{n-10}|.$$
We can also easily check that $|w_n| > 2,008^n$ for all $5\leq n\leq 14$.
This means we have our result
$$|w_n|> 2,008^n\ \textrm{for}\ n\geq 5.$$
\end{proof}

From the basic recursion $b_n$ alone, we get $r(n)\geq 2^n-1$.
Our new recursion gives a slightly better bound $r(n)> 2,008^n$,
which can be improved easily if we do not estimate the length of $w_n$ in Proposition \ref{pro11} that roughly.
We only mention that it can be improved at least to $2,0178^n$, but we will not do it here.

%%%%%%%%%%%%%%%%%%%%%%%%%%%%%%%%%%%%%%%%%%%
%%%%%%%%%%%%%%%%%%%%%%%%%%%%%%%%%%%%%%%%%%%
%%%%%%%%%%%%%%%%%%%%%%%%%%%%%%%%%%%%%%%%%%%

\subsection{An upper bound}\label{upper}

In this subsection, we will prove an upper bound for $r(n)$.
First, we will prove two useful lemmas.
For that, let us mention that every square-free palindrome has to be of odd length,
because palindromes of even length create a square of two letters to the middle, for example 12011021 has a square 11 in the middle.

\begin{lemma}\label{l1}
The middle letter of a rich square-free palindrome is unioccurrent.
\end{lemma}
\begin{proof}
Since all square-free palindromes are of odd length, there always exists the middle letter.
Then, suppose the contrary: $zb\widetilde{z}$ is rich and square-free and the letter $b$ has another occurrence inside $z$.
We can take the other occurrence of $b$ to be consecutive to the $b$ in the middle and get that $zb\widetilde{z}=z_1bz_2bz_2b\widetilde{z_1}$, where $z_2$ is a palindrome because of Proposition \ref{p3}.
We get a contradiction because $bz_2bz_2$ is a square.
\end{proof}

\begin{lemma}\label{l2}
Suppose $w=u_1 a_1 u_2 a_1 \cdots a_1 u_{k-1} a_1 u_k\in\{a_1,a_2,\ldots,a_n\}^*$ is rich and square-free,
where $n,k\geq3$ (possibly $u_k=\epsilon$), $\textrm{Alph}(u_1)=\{a_2,\ldots,a_n\}$ and $\forall i: a_1\notin \textrm{Alph}(u_i)$.
$$For\ 2\leq i\leq k-1:\ \textrm{Alph}(u_{i+1})\subseteq\textrm{Alph}(u_i)\setminus\{a_i\},\ where\ u_i=v_i a_i \widetilde{v_i}.$$
\end{lemma}
\begin{proof}
Since $\forall i: a_1\notin\textrm{Alph}(u_i)$, we get from Proposition \ref{p3} that $u_2,\ldots,u_{k-1}$ are palindromes, and because $w$ is square-free, they are of odd length and non-empty.
By permutating the letters, we can suppose for $2\leq i\leq k-1$: $a_i$ is the middle letter of $u_i=v_i a_i\widetilde{v_i}$, where $a_i\notin\textrm{Alph}(v_i)$ by Lemma \ref{l1}.

We will prove the claim by induction on $i$.

1) The base case $i=2$.
Since $a_2\in\textrm{Alph}(u_1)=\{a_2,\ldots,a_n\}$, we get from Proposition \ref{p3} that $u_1=v_1 a_2 \widetilde{v_2}$.
If $a_2\in\textrm{Alph}(u_3)$ then, by Proposition \ref{p3}, we have $u_3=v_2 a_2 v'_3$, which creates a square $(a_2 \widetilde{v_2} a_1 v_2)^2$ in $u_1 a_1 u_2 a_1 u_3=v_1 a_2 \widetilde{v_2} a_1 v_2 a_2 \widetilde{v_2} a_1 v_2 a_2 v'_3$.
This means $a_2\notin\textrm{Alph}(u_3)$.

Suppose then that $b\in\textrm{Alph}(u_3)\setminus \textrm{Alph}(u_2)$, which implies $b\in\textrm{Alph}(v_1)$.
The word between the first occurrence of $b$ in $u_3$ and the last occurrence of $b$ in $v_1$ is a palindrome by Proposition \ref{p3}:
$u_1 a_1 u_2 a_1 u_3=t_1 b t_2 a_2 \widetilde{v_2} a_1 v_2 a_2\widetilde{v_2} a_1 v_2 a_2 \widetilde{t_2} b t_3$, where $v_1=t_1 b t_2$ and $u_3=v_2 a_2 \widetilde{t_2} b t_3$.
We get a contradiction since we have a square $(a_2 \widetilde{v_2} a_1 v_2)^2$.
This means $\textrm{Alph}(u_3)\subseteq\textrm{Alph}(u_2)\setminus\{a_2\}$.

2) The induction hypothesis. We can now suppose $k\geq4$, since the base case proves our claim if $k=3$.
Suppose then that for every $j$, where $2\leq j\leq i< k-1$, we have: $\textrm{Alph}(u_{j+1})\subseteq\textrm{Alph}(u_j)\setminus\{a_j\}$.

3) The induction step. Now we need to prove that $\textrm{Alph}(u_{i+2})\subseteq\textrm{Alph}(u_{i+1})\setminus\{a_{i+1}\}$.
From the induction hypothesis we get that $a_{i+1}\in\textrm{Alph}(u_{i+1})\subseteq\textrm{Alph}(u_i)\setminus\{a_i\}$,
which means $u_i=v_{i+1} a_{i+1} x a_i \widetilde{x} a_{i+1} \widetilde{v_{i+1}}$ by Proposition \ref{p3}.
If $a_{i+1}\in\textrm{Alph}(u_{i+2})$ then, by Proposition \ref{p3}, we have $u_{i+2}=v_{i+1} a_{i+1} y$,
which creates a square $(a_{i+1} \widetilde{v_{i+1}} a_1 v_{i+1})^2$ inside
$u_i a_1 u_{i+1} a_1 u_{i+2}=v_{i+1} a_{i+1} x a_i \widetilde{x} a_{i+1} \widetilde{v_{i+1}} a_1 v_{i+1}a_{i+1}\widetilde{v_{i+1}} a_1 v_{i+1} a_{i+1} y$.
This means $a_{i+1}\notin\textrm{Alph}(u_{i+2})$

Suppose then that $c\in\textrm{Alph}(u_{i+2})\setminus\textrm{Alph}(u_{i+1})$, which implies $c\in\textrm{Alph}(u_1a_1\ldots a_1 u_i)$.
Without loss of generality, we can assume that $c$ is the letter from $\textrm{Alph}(u_{i+2})\setminus\textrm{Alph}(u_{i+1})$ that has the rightmost occurrence in $u_1a_1\ldots a_1 u_i$.
The word between the leftmost occurrence of $c$ in $u_{i+2}=zcz'$ and the rightmost occurrence of $c$ in $u_1a_1\ldots a_1u_i$ has to be a palindrome by Proposition \ref{p3}.
We divide this into two cases.

- Suppose $c\notin\textrm{Alph}(u_i)$. Now $c\widetilde{z}a_1u_{i+1} P u_{i+1}a_1zc$ is a palindrome, where $\textrm{Alph}(P)\subseteq\textrm{Alph}(a_1 u_{i+1})$ because of the way we chose $c$.
Now the middle letter of the palindrome $a_1u_{i+1} P u_{i+1}a_1$ belongs to $P$ and therefore has other occurrences inside it, in $a_1u_{i+1}$ and in $u_{i+1}a_1$.
This is a contradiction by Lemma \ref{l1}. 

- Suppose $c\in\textrm{Alph}(u_i)$. Now $c\widetilde{z}a_1 v_{i+1} a_{i+1}\widetilde{v_{i+1}}a_1zc$ is a palindrome, where $a_{i+1}$ is its middle letter and $c\widetilde{z}$ is a suffix of $u_i$.
If $a_{i+1}\in\textrm{Alph}(z)$ then it is not unioccurrent in the palindrome $\widetilde{z}a_1 v_{i+1} a_{i+1}\widetilde{v_{i+1}} a_1z$ and we get a contradiction by Lemma \ref{l1}.
Since $a_{i+1}\in\textrm{Alph}(u_i)$ by the induction hypothesis, we can take the rightmost occurrence of it in $u_i$ and get that $a_{i+1} v'_i c \widetilde{z} a_1 v_{i+1} a_{i+1}$ is a palindrome,
where $v'_i c \widetilde{z}=\widetilde{v_{i+1}}$. We get a contradiction since this would mean $c\in\textrm{Alph}(v_{i+1})\subset\textrm{Alph}(u_{i+1})$.

Both cases yield a contradiction, which means $\textrm{Alph}(u_{i+2})\subseteq\textrm{Alph}(u_{i+1})\setminus\{a_{i+1}\}$.
\end{proof}

\begin{corollary}\label{cor1}
All rich square-free words are finite.
\end{corollary}
\begin{proof}
We prove this by induction.
Suppose $w$ is rich and square-free word for which $|\textrm{Alph}(w)|=n\geq4$.
Suppose that all rich square-free words on an alphabet of size $n-1$ or smaller are finite.
Cases $n=1,2,3$ are trivial.

Suppose that $a_1$ is the letter of $w$ for which $w=u_1 a_1 w'$, where $\textrm{Alph}(u_1)=\textrm{Alph}(w)\setminus\{a_1\}$.
We partition $w$ such that $w=u_1 a_1 u_2 a_1 u_3 a_1 u_4 a_1\ldots$, where $a_1\notin\textrm{Alph}(u_i)$ for all $i$.
From Lemma \ref{l2} we now get that $|\textrm{Alph}(u_i)|>|\textrm{Alph}(u_{i+1})|$ for all $i\geq2$.
This means there are finitely many words $u_i$, at most $n$, and they are all over an alphabet of size $n-1$ or smaller, which concludes the proof.
\end{proof}

The above corollary gives another proof for the result mentioned in Remark 6 of \cite{ps}.
The proof of the above corollary also gives us a way to get an upper bound for $r(n)$: $r(n)\leq r(n-1)+1+\sum^{n-1}_{i=1}{(r(n-i)+1)}$.
This bound can be easily improved if we examine the word also from the right side,
i.e. we suppose that $a_1$ is the letter of $w$ for which $w=w' a_1 u_1$, where $\textrm{Alph}(u_1)=\textrm{Alph}(w)\setminus\{a_1\}$.
This notice makes it reasonable to make the following definition.

\begin{definition}
Let $w=uav$ be a word, where $a$ is a letter. If $\textrm{Alph}(u)=\textrm{Alph}(w)\setminus\{a\}$ then the leftmost occurrence of the letter $a$ in $w$ is called the \emph{left special letter} of $w$.
If $\textrm{Alph}(v)=\textrm{Alph}(w)\setminus\{a\}$ then the rightmost occurrence of the letter $a$ in $w$ is called the \emph{right special letter} of $w$.
\end{definition}

In Subsection \ref{lower}, where we constructed the words $w_n$ for our lower bound, the rightmost occurrence of $A_n$ is always the right special letter of $w_n$ and the leftmost occurrence of $B_n$ is always the left special letter of $w_n$, for $n\geq3$.
In Lemma \ref{l2} and Corollary \ref{cor1}, the first occurrence of letter $a_1$ is the left special letter of $w$.

Before we go to our upper bound for $r(n)$, we will state a helpful lemma.

\begin{lemma}\label{l3}
Suppose $w_n=x B_n y A_n z$ is a rich square-free $n$-ary word, where $n\geq3$ and the letters $A_n$ and $B_n$ are the right and left special letters of $w_n$, respectively.
Now $\textrm{Alph}(y)=\textrm{Alph}(w_n)\setminus \{A_n,B_n\}$ and $A_n\neq B_n$.
\end{lemma}
\begin{proof}
First we prove that $A_n,B_n\notin y$. Suppose to the contrary that $B_n \in y$ (case $A_n \in y$ is symmetric).
We can take the leftmost occurrence of $B_n$ in $y$ and get that $w_n = x B_n y_1 c \widetilde{y_1} B_n y_2 A_n z$, where $B_n \notin y_1 c \widetilde{y_1}$ and $c$ is a letter.
Since $A_n$ is the right special letter of $w_n$, we have that $c\in z$.
Since $B_n$ is the left special letter of $w_n$, we get from Lemma \ref{l2} that $c\notin y_2 A_n z$, i.e. $c\notin z$. This is a contradiction.

Then we prove that $A_n\neq B_n$. Suppose to the contrary that $A_n=B_n$.
Now, since $A_n,B_n\notin y$, we get from Proposition \ref{p3} that $y$ is a palindrome.
From Lemma \ref{l2} we get that the middle letter of $y$ cannot be in $x$ nor in $z$.
This is a contradiction, since $x$ and $z$ has to contain all the letters except $A_n$.

Then we prove that if $a\in \textrm{Alph}(w_n)\setminus \{A_n,B_n\}$ then $a\in y$.
Suppose to the contrary that $a\in \textrm{Alph}(w_n)\setminus (\{A_n,B_n\}\cup\textrm{Alph}(y))$.
Since $A_n$ and $B_n$ are the right and left special letters, we have that $a\in x,z$.
If we take the leftmost occurrence of $a$ in $z$ and the rightmost occurrence of $a$ in $x$, then we get from Proposition \ref{p3} that $w= x'a u B_n y A_n v a z'$, where $a u B_n y A_n v a$ is a palindrome, $x=x'a u$ and $z=v a z'$.
The middle letter of the palindrome $a u B_n y A_n v a$ cannot be inside $u$ nor $v$, since it would mean $B_n\in u$ or $A_n\in v$, which is impossible since $A_n$ and $B_n$ are special letters.
The middle letter of $a u B_n y A_n v a$ cannot be inside $y$ neither, since that would mean $B_n\in y A_n$ or $A_n\in B_n y$, which we proved above to be impossible.
The only possibility is that the middle letter of $a u B_n y A_n v a$ is either $A_n$ or $B_n$. Since these cases are symmetric, we can suppose $B_n$ is the middle letter.
This means $w= x'a \widetilde{v} A_n \widetilde{y} B_n y A_n v a z'$. Since $B_n$ is the left special letter of $w$, we have $B_n\in z=v a z'$ and $B_n\notin v$.
This means $B_n\in z'$. If we take the leftmost occurrence of $B_n$ in $z'$, we get $w= x'a \widetilde{v} A_n \widetilde{y} B_n y A_n v a v' B_n z''$, where $B_n y A_n v a v' B_n$ is a palindrome which has $A_n$ as the middle letter.
This means $\widetilde{y}=v a v' $ and hence $a\in y$, which is a contradiction.
\end{proof}

There are only three cases how the right and left special letters can appear inside a word, with respect to each other.
If $w_n$ is a rich square-free $n$-ary word which has $A_n$ and $B_n$ as the right and left special letters, respectively,
then one the following cases must hold (the visible occurrences of $A_n$ and $B_n$ in $w_n$ are the special letters):

1) $w_n = x B_n y A_n z$. Now $A_n\neq B_n$ by Lemma \ref{l3}.

2) $w_n = x A_n y B_n z$. Now $A_n\neq B_n$ by the definition of special letters.

3) $w_n = x A_n z = x B_n z$. Now $A_n=B_n$.

\begin{proposition}\label{pro1}
Suppose $w_n$ is a rich square-free $n$-ary word, where $n\geq3$.

1) If $w_n = x B_n y A_n z$, where the letters $A_n$ and $B_n$ are the right and left special letters of $w_n$, respectively, then $|w_n|\leq 2r(n-1)+r(n-2)+2$.

2) If $w_n = x A_n y B_n z$, where the letters $A_n$ and $B_n$ are the right and left special letters of $w_n$, respectively, then $|w_n|\leq r(n-1)+r(n-2)+r(n-3)+2\leq 2r(n-1)$ and $|x|,|z|\leq r(n-2)+r(n-3)+1$, where $r(n-3)=0$ if $n=3$.

3) If $w_n = x A_n z = x B_n z$, where the letter $A_n=B_n$ is both the right and left special letter of $w_n$, then $|w_n|\leq 2r(n-1)+1$.
\end{proposition}
\begin{proof}
Let us denote $A=\textrm{Alph}(w_n)$.

1) By the definition of special letters, we have that $\textrm{Alph}(x)=A\setminus\{B_n\}$ and $\textrm{Alph}(z)=A\setminus\{A_n\}$. These mean $|x|,|z|\leq r(n-1)$.
From Lemma \ref{l3} we get that $\textrm{Alph}(y)=A\setminus\{A_n,B_n\}$, which means $|y|\leq r(n-2)$, since $A_n\neq B_n$.
Now $$|w_n|=|x|+|B_n|+|y|+|A_n|+|z|\leq r(n-1)+1+r(n-2)+1+r(n-1)=2r(n-1)+r(n-2)+2.$$

2) If $A_n\notin x$, then $|x|\leq r(n-2)$. If $A_n\in x$ then we can take the rightmost occurrence of it in $x$ and get that $xA_n=x_2 A_n x_1 c \widetilde{x_1} A_n$, where $A_n\notin x_1 c \widetilde{x_1}$ and by Lemma \ref{l2} $c\notin x_2 A_n x_1$.
Now $\textrm{Alph}(x_2 A_n x_1)=A\setminus\{c,B_n\}$ and $\textrm{Alph}(\widetilde{x_1})=A\setminus\{c,A_n,B_n\}$, where $c\neq B_n$ since $B_n$ is the left special letter of $w_n$.
This means $|x|=|x_2 A_n x_1|+|c|+|\widetilde{x_1}|\leq r(n-2)+r(n-3)+1$, where $r(n-3)=0$ if $n=3$.
The same holds for $z$.

We have $\textrm{Alph}(y B_n z)=A\setminus\{A_n\}$, which means $|y B_n z|\leq r(n-1)$.
Now $$|w_n|=|x|+|A_n|+|y B_n z|\leq [r(n-2)+r(n-3)+1]+1+r(n-1)=r(n-1)+r(n-2)+r(n-3)+2.$$

From the basic recursion we know that $r(n)\geq 2 r(n-1)+1$.
This means that $r(n-1)+r(n-2)+r(n-3)+2\leq r(n-1)+r(n-2)+2r(n-3)+2 \leq r(n-1)+2r(n-2)+1\leq 2r(n-1)$, which we needed to prove.

3) By the definition of special letters, we have that $\textrm{Alph}(x)=\textrm{Alph}(z)=A\setminus\{A_n\}$, which means $|x|,|z|\leq r(n-1)$.
Now $$|w_n|=|x|+|A_n|+|z|\leq r(n-1)+1+r(n-1)=2r(n-1)+1.$$
\end{proof}

\begin{corollary}\label{cor2}
$r(n)\leq 2r(n-1)+r(n-2)+2$, for $n\geq3$.
\end{corollary}
\begin{proof}
We get our claim from Proposition \ref{pro1}, since the proposition covered all the three different possible cases for $w_n$.
\end{proof}

We do not solve the recursion $r(n)\leq 2r(n-1)+r(n-2)+2, r(2)=3, r(1)=1$, in a closed-form, but we will estimate it.
We use the inequality $r(n)\geq 2 r(n-1)+1$ from the basic recursion, and the fact that $r(4)=15>13$.
For $n\geq8$ we have 
$$r(n)\leq 2r(n-1)+r(n-2)+2 \leq 2(2r(n-2)+r(n-3)+2)+r(n-2)+2 = 5r(n-2)+2r(n-3)+6$$
$$\leq 5(2r(n-3)+r(n-4)+2)+2r(n-3)+6 = 12r(n-3)+5r(n-4)+16$$
$$< 12r(n-3)+5r(n-4)+16+(r(n-4)-13) = 12r(n-3)+6r(n-4)+3 \leq 15r(n-3)$$
$$< 2,47^3r(n-3) < 2,47^n,$$
where the last inequality comes from the fact that $r(n)<2,47^n$ for $1\leq n\leq 7$.
Together with the lower bound, we now have $2,008^n<r(n)<2,47^n$ for $n\geq5$.

This upper bound can still be improved. The cases 2 and 3 from Proposition \ref{pro1} already give better or equal upper bounds than the basic recursion, i.e. $r(n)\leq 2r(n-1)+1$.
This means we need to look closer only for the case 1.

\begin{proposition}\label{pro2}
$r(n)\leq 5r(n-2)+4$, for $n\geq7$.
\end{proposition}
\begin{proof}
Suppose $w_n=x B_n y A_n z$ is a rich square-free $n$-ary word, where $n\geq7$ and the letters $A_n$ and $B_n$ are the right and left special letters of $w_n$, respectively. This means $A_n\neq B_n$.
If $w_n$ is not of this form, then we already know from Proposition \ref{pro1} that $|w_n|\leq 2r(n-1)+1$,
which means we can use the upper bound of Corollary \ref{cor2} and get that
$$|w_n|\leq 2(2r(n-2)+r(n-3)+2)+1=4r(n-2)+2r(n-3)+5\leq 5r(n-2)+4,$$
where the last inequality comes from the basic recursion $r(n)\geq 2 r(n-1)+1$.
From now on, we will use the basic recursion without mentioning it.

By the definition of special letters, we have that $A_n\in x$ and $B_n\in z$. From Lemma \ref{l3} we know that $A_n,B_n\notin y$. Since $A_n\neq B_n$, we can take the rightmost
occurrence of $A_n$ in $x$ and the leftmost occurrence of $B_n$ in $z$ and get, by Proposition \ref{p3}, that $w_n = x_1 A_n \widetilde{y} B_n y A_n \widetilde{y} B_n z_1$.

We divide this proof into three different cases depending whether $A_n\in x_1$ or $A_n\notin x_1$ and whether $B_n\in z_1$ or $B_n\notin z_1$.

%%%%%%%%%%%%%%%%%%%%%%%%%%%%%%%%%%%%%%%%%%%%%%%%%%%%%%%%%%%%%%%%%%%%%%%
%%%%%%%%%%%%%%%%%%%%%%%%%%%%%%%%%%%%%%%%%%%%%%%%%%%%%%%%%%%%%%%%%%%%%%%
%%%%%%%%%%%%%%%%%%%%%%%%%%%%%%%%%%%%%%%%%%%%%%%%%%%%%%%%%%%%%%%%%%%%%%%

Case 1) $A_n\notin x_1, B_n\notin z_1$.

Now we have $A_n,B_n\notin x_1,z_1,y$. This means $|x_1|,|z_1|,|y|\leq r(n-2)$. Together we get
$$|w_n|=|x_1 A_n\widetilde{y} B_n y A_n \widetilde{y} B_n z_1|\leq 5r(n-2)+4.$$

%%%%%%%%%%%%%%%%%%%%%%%%%%%%%%%%%%%%%%%%%%%%%%%%%%%%%%%%%%%%%%%%%%%%%%%
%%%%%%%%%%%%%%%%%%%%%%%%%%%%%%%%%%%%%%%%%%%%%%%%%%%%%%%%%%%%%%%%%%%%%%%
%%%%%%%%%%%%%%%%%%%%%%%%%%%%%%%%%%%%%%%%%%%%%%%%%%%%%%%%%%%%%%%%%%%%%%%

Case 2) $A_n\in x_1, B_n\notin z_1$ (the case $A_n\notin x_1, B_n\in z_1$ is symmetric).

If we take the rightmost occurrence of $A_n$ in $x_1$ we get, by Proposition \ref{p3}, Lemma \ref{l1} and Lemma \ref{l2}, that $w_n = x_2 A_n \widetilde{x_B} B x_B A_n \widetilde{y} B_n y A_n \widetilde{y} B_n z_1$,
where $B\ (\neq A_n,B_n)$ is a letter, $A_n,B\notin x_B$, $B\notin x_2$ and $x_1=x_2 A_n \widetilde{x_B} B x_B$.
Since $B_n$ is a left special letter of $w_n$, we have that $B_n\notin x_2 A_n \widetilde{x_B}$ and $B_n\notin x_B$.
We also have $A_n,B_n\notin y,z_1$.
Together we have $|y|,|z_1|,|x_2 A_n \widetilde{x_B}|\leq r(n-2)$ and $|x_B|\leq r(n-3)$.

Let us mark the left special letter of $\widetilde{y}$ with $B_{n-2}$.
Now we divide this into two cases whether $B\neq B_{n-2}$ or $B = B_{n-2}$.

Case 2.1) $B\neq B_{n-2}$.

Since $B_{n-2}$ is the left special letter of $\widetilde{y}$, we must have $B_{n-2}\notin x_B$. Otherwise we would have, by Proposition \ref{p3}, that $B\in x_B$, which is impossible by Lemma \ref{l1}.
From Lemma \ref{l2} we now get that $B_{n-2}\notin x_2$.
Earlier, we already noted that $B_n,B\notin x_2 A_n \widetilde{x_B}$ and $A_n,B_n,B\notin x_B$.
Together we now get $|x_2 A_n \widetilde{x_B}|\leq r(n-3)$ and $|x_B|\leq r(n-4)$, and therefore
$$|w_n|=|x_2 A_n \widetilde{x_B}|+|B|+|x_B|+|A_n \widetilde{y} B_n y A_n \widetilde{y} B_n|+|z_1|$$
$$\leq r(n-3)+1+r(n-4)+[3r(n-2)+4]+r(n-2) = 4r(n-2)+r(n-3)+r(n-4)+5$$
$$< 4r(n-2)+r(n-3)+r(n-4)+5+r(n-4) \leq 5r(n-2)+3,$$
where we added the extra $r(n-4)$ after the second inequality only to make the use of the basic recursion simpler.

Case 2.2) $B = B_{n-2}$.

If we can prove that $|z_1|\leq r(n-3)$, then we get
$$|w_n|=|x_2 A_n \widetilde{x_B}|+|B|+|x_B|+|A_n\widetilde{y}B_n y A_n \widetilde{y} B_n|+|z_1|$$
$$\leq r(n-2)+1+r(n-3)+[3r(n-2)+4]+r(n-3) = 4r(n-2)+2r(n-3)+5\leq 5r(n-2)+4.$$
So we need to prove there exists some letter, different from $A_n$ and $B_n$, such that it does not belong to $z_1$.
We divide this into three cases depending of which form $\widetilde{y}$ is.

Case 2.2.1) $\widetilde{y}=y_1 A_{n-2} y_3 B_{n-2} y_2$, where the letters $A_{n-2}$ and $B_{n-2}$ are the right and left special letters of $\widetilde{y}$, respectively.

Because $B = B_{n-2}$, we have $\widetilde{x_B} = y_1 A_{n-2} y_3$, by Proposition \ref{p3} and Lemma \ref{l1}.
Now $A_{n-2}\notin z_1$, since otherwise we could take the leftmost occurrence of $A_{n-2}$ in $z_1$ and get a square in $w_n$:
$$ \widetilde{y_1} A_n \widetilde{y} B_n \widetilde{y_2} B_{n-2} \widetilde{y_3} A_{n-2}
   \widetilde{y_1} A_n \widetilde{y} B_n \widetilde{y_2} B_{n-2} \widetilde{y_3} A_{n-2},$$
where the rightmost $\widetilde{y_2} B_{n-2} \widetilde{y_3} A_{n-2}$ is a prefix of $z_1$ and the leftmost $\widetilde{y_1}$ is a suffix $x_1$.

Case 2.2.2) $\widetilde{y}=y_1 B_{n-2} y_2$, where $B_{n-2}$ is also the right special letter of $\widetilde{y}$.

Because $B = B_{n-2}$, we have $\widetilde{x_B} = y_1$.
Now $B_{n-2}\notin z_1$, since otherwise, similar to Case 2.2.1, we could take the leftmost occurrence of $B_{n-2}$ in $z_1$ and get a square in $w_n$:
$$ \widetilde{y_1} A_n \widetilde{y} B_n \widetilde{y_2} B_{n-2}
   \widetilde{y_1} A_n \widetilde{y} B_n \widetilde{y_2} B_{n-2}.$$

Case 2.2.3) $\widetilde{y}=y_1 A_{n-2} y_3 B_{n-2} \widetilde{y_3} A_{n-2} y_3 B_{n-2} y_2$, where the rightmost $A_{n-2}$ and the leftmost $B_{n-2}$ are the right and left special letters of $\widetilde{y}$, respectively.

Because $B = B_{n-2}$, we have $\widetilde{x_B} = y_1 A_{n-2} y_3$.
Again $A_{n-2}\notin z_1$, since otherwise, similar to Case 2.2.1, we could take the leftmost occurrence of $A_{n-2}$ in $z_1$ and get a square in $w_n$:
$$ y_3 B_{n-2} \widetilde{y_3} A_{n-2} \widetilde{y_1} A_n \widetilde{y} B_n \widetilde{y_2} B_{n-2} \widetilde{y_3} A_{n-2}
   y_3 B_{n-2} \widetilde{y_3} A_{n-2} \widetilde{y_1} A_n \widetilde{y} B_n \widetilde{y_2} B_{n-2} \widetilde{y_3} A_{n-2}.$$

%%%%%%%%%%%%%%%%%%%%%%%%%%%%%%%%%%%%%%%%%%%%%%%%%%%%%%%%%%%%%%%%%%%%%%%
%%%%%%%%%%%%%%%%%%%%%%%%%%%%%%%%%%%%%%%%%%%%%%%%%%%%%%%%%%%%%%%%%%%%%%%
%%%%%%%%%%%%%%%%%%%%%%%%%%%%%%%%%%%%%%%%%%%%%%%%%%%%%%%%%%%%%%%%%%%%%%%

Case 3) $A_n\in x_1, B_n\in z_1$.

If we take the rightmost occurrence of $A_n$ in $x_1$ and the leftmost occurrence of $B_n$ in $z_1$,
we get that $w_n = x_2 A_n \widetilde{x_B} B x_B A_n \widetilde{y} B_n y A_n \widetilde{y} B_n z_A A \widetilde{z_A} B_n z_2$, where $A,B\ (\neq A_n,B_n)$ are letters and $x_1=x_2 A_n \widetilde{x_B} B x_B$, $z_1=z_A A \widetilde{z_A} B_n z_2$.
Similar to Case 2, we have $|y|,|x_1|,|z_1|,|x_2 A_n \widetilde{x_B}|$,
$|\widetilde{z_A} B_n z_2|\leq r(n-2)$ and $|x_B|,|z_A|\leq r(n-3)$.

We divide this case now into three cases depending of which form $\widetilde{y}$ is.

Case 3.1) $\widetilde{y}=y_1 B_{n-2} y_2$, where $B_{n-2}$ is both the right and left special letter of $\widetilde{y}$.

If $A=B=B_{n-2}$ then $x_B=\widetilde{y_1}$ and $z_A=\widetilde{y_2}$. This would create a square in $w_n$:
$$B_{n-2} \widetilde{y_1} A_n \widetilde{y} B_n \widetilde{y_2}
  B_{n-2} \widetilde{y_1} A_n \widetilde{y} B_n \widetilde{y_2}.$$
Now we divide this into two possible cases: $A,B\neq B_{n-2}$ and $A=B_{n-2}$,$B\neq B_{n-2}$.

Case 3.1.1) $A,B\neq B_{n-2}$.

Similar to Case 2.1, we get $B_n,B_{n-2},B\notin x_2 A_n \widetilde{x_B}$ and $A_n,B_n,B_{n-2},B\notin x_B$.
In the same way, we get $A_n,A,B_{n-2}\notin \widetilde{z_A} B_n z_2$ and $A_n,A,B_n,B_{n-2}\notin z_A$.
Together we have
$$|w_n|=|x_2 A_n \widetilde{x_B}|+|B|+|x_B|+|A_n \widetilde{y} B_n y A_n \widetilde{y} B_n|+|z_A|+|A|+|\widetilde{z_A} B_n z_2|$$
$$\leq r(n-3)+1+r(n-4)+[3r(n-2)+4]+r(n-4)+1+r(n-3)=3r(n-2)+2r(n-3)+2r(n-4)+6$$
$$< 3r(n-2)+2r(n-3)+2r(n-4)+6+2r(n-4)\leq 5r(n-2)+2.$$

Case 3.1.2) $A=B_{n-2}$ and $B\neq B_{n-2}$ (the case $A\neq B_{n-2}$ and $B=B_{n-2}$ is symmetric).

Now $z_A=\widetilde{y_2}$. Let us mark $y_1=u_1 B_{n-4} u_2$ and $y_2=v_1 A_{n-4} v_2$, where $B_{n-4}$ and $A_{n-4}$ are the left special letters of $y_1$ and $y_2$, respectively.

We prove $B\neq B_{n-4}$. Suppose to the contrary that $B=B_{n-4}$.
Since $B_{n-2}$ is the right and left special letter of $\widetilde{y}$, we have that $A_{n-4}\in y_1$.
If we take the rightmost occurrence of $A_{n-4}$ in $y_1$ then we get from Proposition \ref{p3} that $A_{n-4}\widetilde{v_1}$ is a suffix of $y_1$ and hence $A_{n-4}$ is the right special letter of $y_1$.
There are now three different cases how $A_{n-4}$ and $B_{n-4}$ can appear inside $y_1$ with respect to each other. These all yield a square and hence a contradiction:

- If $y_1=u'_1 A_{n-4} u_3 B_{n-4} \widetilde{u_3} A_{n-4} u_3 B_{n-4} u'_2$, where $u_1=u'_1 A_{n-4} u_3$, $u_2=\widetilde{u_3} A_{n-4} u_3 B_{n-4} u'_2$ and $v_1=\widetilde{u'_2} B_{n-4} \widetilde{u_3}$, then we have a square in $w_n$:
$$A_{n-4} u_3 B_{n-4} \widetilde{u_3} A_{n-4} \widetilde{u'_1} A_n \widetilde{y} B_n \widetilde{y_2} B_{n-2} \widetilde{u'_2} B_{n-4} \widetilde{u_3}
  A_{n-4} u_3 B_{n-4} \widetilde{u_3} A_{n-4} \widetilde{u'_1} A_n \widetilde{y} B_n \widetilde{y_2} B_{n-2} \widetilde{u'_2} B_{n-4} \widetilde{u_3}.$$

- If $y_1 =u_1 B_{n-4} u_2=u_1 A_{n-4} \widetilde{v_1}$ (i.e. $A_{n-4}=B_{n-4}$), then $u_2=\widetilde{v_1}$ and we have a square in $w_n$:
$$B_{n-4} \widetilde{u_1} A_n \widetilde{y} B_n \widetilde{y_2} B_{n-2} \widetilde{u_2}
  B_{n-4} \widetilde{u_1} A_n \widetilde{y} B_n \widetilde{y_2} B_{n-2} \widetilde{u_2}.$$

- If $y_1=u'_1 A_{n-4} u_3 B_{n-4} u_2$, where $u_1=u'_1 A_{n-4} u_3$ and $v_1=\widetilde{u_2} B_{n-4} \widetilde{u_3}$, then we have a square in $w_n$:
$$B_{n-4} \widetilde{u_3} A_{n-4} \widetilde{u'_1} A_n \widetilde{y} B_n \widetilde{y_2} B_{n-2} \widetilde{u_2}
  B_{n-4} \widetilde{u_3} A_{n-4} \widetilde{u'_1} A_n \widetilde{y} B_n \widetilde{y_2} B_{n-2} \widetilde{u_2}.$$

This means $B\neq B_{n-4}$. Similar to Case 2.1 we now get that $B_n,B_{n-2},B_{n-4},B\notin x_2 A_n \widetilde{x_B}$ and $A_n,B_n,B_{n-2},B_{n-4},B\notin x_B$.
Together we have
$$|w_n|=|x_2 A_n \widetilde{x_B}|+|B|+|x_B|+|A_n \widetilde{y} B_n y A_n \widetilde{y} B_n|+|z_A|+|B_{n-2}|+|\widetilde{z_A} B_n z_2|$$
$$\leq r(n-4)+1+r(n-5)+[3r(n-2)+4]+r(n-3)+1+r(n-2)$$
$$= 4r(n-2)+r(n-3)+r(n-4)+r(n-5)+6$$
$$< 4r(n-2)+r(n-3)+r(n-4)+r(n-5)+6+r(n-5)\leq 5r(n-2)+3.$$

%%%%%%%%%%%%%%%%%%%%%%%%%%%%%%%%%%%%%%%%%%%%%%%%%%%%%%%%%%%%%%%%%%%%%%%

Case 3.2) $\widetilde{y}=y_1 A_{n-2} y_3 B_{n-2} y_2$, where the letters $A_{n-2}$ and $B_{n-2}$ are the right special letter and the left special letter of $\widetilde{y}$, respectively.

If $A=A_{n-2}$, $B=B_{n-2}$ then we would have a square in $w_n$:
$$A_{n-2} \widetilde{y_1} A_n \widetilde{y} B_n \widetilde{y_2} B_{n-2} \widetilde{y_3}
  A_{n-2} \widetilde{y_1} A_n \widetilde{y} B_n \widetilde{y_2} B_{n-2} \widetilde{y_3}.$$
This means we can divide this case, similar to Case 3.1, into two different cases: $A=A_{n-2}$, $B\neq B_{n-2}$ and $A\neq A_{n-2}$, $B\neq B_{n-2}$.

Case 3.2.1) $A\neq A_{n-2}$, $B\neq B_{n-2}$.

Similar to Case 2.1, we get $B_n,B_{n-2},B\notin x_2 A_n \widetilde{x_B}$ and $A_n,B_n,B_{n-2},B\notin x_B$.
In the same way, we get $A_n,A_{n-2},A\notin \widetilde{z_A} B_n z_2$ and $A_n,A_{n-2},A,B_n\notin z_A$.
Together we have
$$|w_n|=|x_2 A_n \widetilde{x_B}|+|B|+|x_B|+|A_n \widetilde{y} B_n y A_n \widetilde{y} B_n|+|z_A|+|A|+|\widetilde{z_A} B_n z_2|$$
$$\leq r(n-3)+1+r(n-4)+[3r(n-2)+4]+r(n-4)+1+r(n-3)=3r(n-2)+2r(n-3)+2r(n-4)+6$$
$$< 3r(n-2)+2r(n-3)+2r(n-4)+6+2r(n-4)\leq 5r(n-2)+2.$$

Case 3.2.2) $A=A_{n-2}$, $B\neq B_{n-2}$ (the case $A\neq A_{n-2}$,$B=B_{n-2}$ is symmetric).

Now $z_A = \widetilde{y_2} B_{n-2} \widetilde{y_3}$.
We divide this case into two cases: $A_{n-2}\notin y_1$ and $A_{n-2}\in y_1$.

Case 3.2.2.1) $A_{n-2}\notin y_1$.

We must have $A_{n-2}\notin x_1$. Otherwise we could take the rightmost occurrence of $A_{n-2}$ in $x_1$ and get a square in $w_n$:
$$A_{n-2} \widetilde{y_1} A_n \widetilde{y} B_n \widetilde{y_2} B_{n-2} \widetilde{y_3} A_{n-2} \widetilde{y_1} A_n \widetilde{y} B_n \widetilde{y_2} B_{n-2} \widetilde{y_3}.$$

Similar to Case 2.1, we have $B_n,B_{n-2}\notin x_1$.
Since $B_n$ and $B_{n-2}$ are the left special letters of $w_n$ and $\widetilde{y}$, respectively, we have $B_n,B_{n-2}\notin y_1$.
Together with the previous paragraph we get that $A_{n-2},B_n,B_{n-2}\notin x_1 A_n y_1$.
Since $A_{n-2}$ is the right special letter of $\widetilde{y}$ we have $A_n,B_n,A_{n-2}\notin y_3 B_{n-2} y_2$.
These mean $|x_1 A_n y_1|\leq r(n-3)$ and $|y_3 B_{n-2} y_2|\leq r(n-3)$.
Together we have
$$|w_n|=|x_1 A_n y_1|+|A_{n-2}|+|y_3 B_{n-2} y_2|+|B_n y A_n \widetilde{y} B_n|+|z_A|+|A_{n-2}|+|\widetilde{z_A} B_n z_2|$$
$$\leq r(n-3)+1+r(n-3)+[2r(n-2)+3]+r(n-3)+1+r(n-2)$$
$$= 3r(n-2)+3r(n-3)+5 < 3r(n-2)+3r(n-3)+5+r(n-3) \leq 5r(n-2)+3.$$

Case 3.2.2.2) $A_{n-2}\in y_1$.

If we take the rightmost occurrence of $A_{n-2}$ in $y_1$, we get $\widetilde{y}=y'_1 A_{n-2} y_4 B_y \widetilde{y_4} A_{n-2} y_3 B_{n-2} y_2$,
where $B_y$ is a letter, $y_1=y'_1 A_{n-2} y_4 B_y \widetilde{y_4}$ and $A_{n-2}\notin y_4 B_y \widetilde{y_4}$.
Let us mark $y_3= u_1 B_{n-4} u_2$, where $B_{n-4}$ is the left special letter of $y_3$.
We will prove $B_{n-4}\notin x_1$.

Suppose $B_{n-4}\notin y_1$. Now $B_{n-4}\notin x_1$, since otherwise we could take the rightmost occurrence of $B_{n-4}$ in $x_1$ and get a square in $w_n$:
$$A_{n-2} \widetilde{y_1}  A_n \widetilde{y} B_n \widetilde{y_2} B_{n-2} \widetilde{y_3} A_{n-2} \widetilde{y_1} A_n \widetilde{y} B_n \widetilde{y_2} B_{n-2} \widetilde{y_3}.$$
Suppose $B_{n-4}\in y_1$. Because of Lemma \ref{l1}, we must have that $B_y=B_{n-4}$ and $y_4=\widetilde{u}_1$. Also now $B_{n-4}\notin x_1$, since otherwise we would have a square in $w_n$:
$$B_{n-4} \widetilde{u}_1 A_{n-2} \widetilde{y'}_1 A_n \widetilde{y} B_n \widetilde{y}_2 B_{n-2} \widetilde{y_3} A_{n-2} \widetilde{u}_1
  B_{n-4} \widetilde{u}_1 A_{n-2} \widetilde{y'}_1 A_n \widetilde{y} B_n \widetilde{y}_2 B_{n-2} \widetilde{y_3} A_{n-2} \widetilde{u}_1.$$
This means we have $B_{n-4}\notin x_1$.

If $B_y=B_{n-4}$ then we get from Lemma \ref{l2} that $B_{n-4}\notin y'_1 A_{n-2} y_4$.
If $B_y\neq B_{n-4}$ then, since $B_{n-4}$ is the left special letter of $y_3$, we also get from Lemma \ref{l1} and \ref{l2} that $B_{n-4}\notin y'_1 A_{n-2} y_4$.
These mean $B_{n-4}\notin x_1 A_n y'_1 A_{n-2} y_4$.

From Lemma \ref{l1} we get that $B_y\notin\widetilde{y_4}$, which means $A_n,A_{n-2},B_n,B_{n-2},B_y\notin\widetilde{y_4}$.
Since $A_{n-2}$ is the right special letter of $\widetilde{y}$, we have that $A_n,A_{n-2},B_n\notin y_3 B_{n-2} y_2$.
Together we have
$$|w_n|=|x_1 A_n y'_1 A_{n-2} y_4| + |B_y| + |\widetilde{y_4}|+|A_{n-2}|+|y_3 B_{n-2} y_2|+|B_n y A_n \widetilde{y} B_n|+|z_A|+|A_{n-2}|+|\widetilde{z_A} B_n z_2|$$
$$\leq r(n-3)+1+r(n-5)+1+r(n-3)+[2r(n-2)+3]+r(n-3)+1+r(n-2)$$
$$= 3r(n-2)+3r(n-3)+r(n-5)+6 < 3r(n-2)+3r(n-3)+r(n-5)+6+3r(n-5) \leq 5r(n-2)+1.$$

%%%%%%%%%%%%%%%%%%%%%%%%%%%%%%%%%%%%%%%%%%%%%%%%%%%%%%%%%%%%%%%%%%%%%%%  

Case 3.3) $\widetilde{y}=y_1 A_{n-2} y_3 B_{n-2} \widetilde{y_3} A_{n-2} y_3 B_{n-2} y_2$, where the rightmost $A_{n-2}$ and the leftmost $B_{n-2}$ are the right and left special letters of $\widetilde{y}$, respectively.

If $A=A_{n-2}$, $B=B_{n-2}$ then we would have a square in $w_n$:
$$B_{n-2} \widetilde{y_3} A_{n-2} \widetilde{y_1} A_n \widetilde{y} B_n \widetilde{y_2} B_{n-2} \widetilde{y_3} A_{n-2} y_3
  B_{n-2} \widetilde{y_3} A_{n-2} \widetilde{y_1} A_n \widetilde{y} B_n \widetilde{y_2} B_{n-2} \widetilde{y_3} A_{n-2} y_3.$$
This means we can divide this case, similar to Case 3.1, into two different cases: $A=A_{n-2}$, $B\neq B_{n-2}$ and $A\neq A_{n-2}$, $B\neq B_{n-2}$.

Case 3.3.1) $A\neq A_{n-2}$ and $B\neq B_{n-2}$.

Similar to Case 2.1, we get $B_n,B_{n-2},B\notin x_2 A_n \widetilde{x_B}$ and $A_n,B_n,B_{n-2},B\notin x_B$.
In the same way, we get $A_n,A_{n-2},A\notin \widetilde{z_A} B_n z_2$ and $A_n,A_{n-2},A,B_n\notin z_A$.
Together we have
$$|w_n|=|x_2 A_n \widetilde{x_B}|+|B|+|x_B|+|A_n \widetilde{y} B_n y A_n \widetilde{y} B_n|+|z_A|+|A|+|\widetilde{z_A} B_n z_2|$$
$$\leq r(n-3)+1+r(n-4)+[3r(n-2)+4]+r(n-4)+1+r(n-3)=3r(n-2)+2r(n-3)+2r(n-4)+6$$
$$< 3r(n-2)+2r(n-3)+2r(n-4)+6+2r(n-4)\leq 5r(n-2)+2.$$

Case 3.3.2) $A=A_{n-2}$, $B\neq B_{n-2}$ (the case $A\neq A_{n-2}$, $B=B_{n-2}$ is symmetric).

Let $A_{n-4}$ be the right special letter of $y_3$.
We will divide this into two cases: $A_{n-4}\notin y_2$ and $A_{n-4}\in y_2$.

Case 3.3.2.1) $A_{n-4}\notin y_2$.

If $A_{n-4}\in z_2$ then we could take the leftmost occurrence of it in $z_2$, which would create a square in $w_n$:
$$ \widetilde{y_3} A_{n-2} y_3 B_{n-2} y_2 B_n \widetilde{y_2} B_{n-2}
   \widetilde{y_3} A_{n-2} y_3 B_{n-2} y_2 B_n \widetilde{y_2} B_{n-2}.$$
This means $A_{n-4}\notin z_2$.
Let us now mark $y_3 = u_1 A_{n-4} u_2$, where the letter $A_{n-4}$ is the right special letter.
We get that $A_{n-4}\notin u_2 B_{n-2} y_2 B_n z_2$. Similar to Case 2.1, we also have $A_n,A_{n-2}\notin u_2 B_{n-2} y_2 B_n z_2$.
From Proposition \ref{pro1} we get that $|u_1|\leq r(n-5)+r(n-6)+1$.
Similar to Case 2.1, we get $B_n,B_{n-2},B\notin x_2 A_n \widetilde{x_B}$ and $A_n,B_n,B_{n-2},B\notin x_B$.
Together we have
$$|w_n|=|x_2 A_n \widetilde{x_B}|+|B|+|x_B|+|A_n \widetilde{y} B_n y A_n \widetilde{y} B_n|+|\widetilde{y_2} B_{n-2} \widetilde{u_2}|+|A_{n-4}\widetilde{u_1}A_{n-2}u_1A_{n-4}|+|u_2 B_{n-2} y_2 B_n z_2|$$
$$\leq r(n-3)+1+r(n-4)+[3r(n-2)+4]+r(n-4)+[2(r(n-5)+r(n-6)+1)+3]+r(n-3)$$
$$ = 3r(n-2)+2r(n-3)+2r(n-4)+2r(n-5)+2r(n-6)+10$$
$$ < 3r(n-2)+2r(n-3)+2r(n-4)+2r(n-5)+2r(n-6)+10+2r(n-6)\leq 5r(n-2)+2.$$

Case 3.3.2.2) $A_{n-4}\in y_2$.

We will divide this case into three cases depending of which form $y_3$ is.

Case 3.3.2.2.1) $y_3 = u_1 A_{n-4} u_3 B_{n-4} u_2$, where $A_{n-4}$ and $B_{n-4}$ are the right and left special letters of $y_3$, respectively.

Since $A_{n-4}\in y_2$, we have that $y_2= \widetilde{u_2} B_{n-4} \widetilde{u_3} A_{n-4} y'_2$, where the $A_{n-4}$ is the leftmost occurrence of $A_{n-4}$ in $y_2$.
If $B_{n-4}\in y_1$ then $y_1=y'_1 B_{n-4} \widetilde{u_3} A_{n-4} \widetilde{u_1}$, where the $B_{n-4}$ is the rightmost occurrence of $B_{n-4}$ in $y_1$.
This would create a square in $\widetilde{y}$:
$$B_{n-4} \widetilde{u_3} A_{n-4} \widetilde{u_1} A_{n-2} y_3 B_{n-2} \widetilde{u_2}
  B_{n-4} \widetilde{u_3} A_{n-4} \widetilde{u_1} A_{n-2} y_3 B_{n-2} \widetilde{u_2}.$$
So $B_{n-4}\notin y_1$.
Now, if $B=B_{n-4}$ then $x_B=\widetilde{u_3} A_{n-4} \widetilde{u_1} A_{n-2} \widetilde{y_1}$ by Lemma \ref{l1}, since $B_{n-4}\notin y_1$. This would create a square in $w_n$:
$$B_{n-4} \widetilde{u_3} A_{n-4} \widetilde{u_1} A_{n-2} \widetilde{y_1} A_n \widetilde{y} B_n \widetilde{y_2} B_{n-2} \widetilde{y_3} A_{n-2} y_3 B_{n-2} \widetilde{u_2}$$
$$B_{n-4} \widetilde{u_3} A_{n-4} \widetilde{u_1} A_{n-2} \widetilde{y_1} A_n \widetilde{y} B_n \widetilde{y_2} B_{n-2} \widetilde{y_3} A_{n-2} y_3 B_{n-2} \widetilde{u_2}.$$
So $B\neq B_{n-4}$. This means that, in similar way as in Case 2.1, we get $B_n,B_{n-2},B_{n-4},B\notin x_2 A_n \widetilde{x_B}$ and $A_n,B_n,B_{n-2},B_{n-4},B\notin x_B$.
Together we have
$$|w_n|=|x_2 A_n \widetilde{x_B}|+|B|+|x_B|+|A_n \widetilde{y} B_n y A_n \widetilde{y} B_n|+|z_A|+|A_{n-2}|+|\widetilde{z_A} B_n z_2|$$
$$\leq r(n-4)+1+r(n-5)+[3r(n-2)+4]+r(n-3)+1+r(n-2)$$
$$= 4r(n-2)+r(n-3)+r(n-4)+r(n-5)+5$$
$$ < 4r(n-2)+r(n-3)+r(n-4)+r(n-5)+5+r(n-5)\leq 5r(n-2)+2.$$

Case 3.3.2.2.2) $y_3 = u_1 B_{n-4} u_2$, where $B_{n-4}$ is both the right and left special letter.

This case is very similar to the previous, Case 3.3.2.2.1.

Now $B_{n-4}$ is both the right and left special letter, which means $A_{n-4}=B_{n-4}$.
Since this case is a subcase of Case 3.3.2.2, we have that $A_{n-4}=B_{n-4}\in y_2$, which means $y_2= \widetilde{u_2} B_{n-4} y'_2$.
If $B_{n-4}\in y_1$ then $y_1=y'_1 B_{n-4} \widetilde{u_1}$ and we would have a square in $\widetilde{y}$:
$$B_{n-4} \widetilde{u_1} A_{n-2} y_3 B_{n-2} \widetilde{u_2}
  B_{n-4} \widetilde{u_1} A_{n-2} y_3 B_{n-2} \widetilde{u_2}.$$
So $B_{n-4}\notin y_1$.
If $B=B_{n-4}$ then $x_B=\widetilde{u_1} A_{n-2} \widetilde{y_1}$. This would create a square in $w_n$:
$$B_{n-4} \widetilde{u_1} A_{n-2} \widetilde{y_1} A_n \widetilde{y} B_n \widetilde{y_2} B_{n-2} \widetilde{y_3} A_{n-2} y_3 B_{n-2} \widetilde{u_2}$$
$$B_{n-4} \widetilde{u_1} A_{n-2} \widetilde{y_1} A_n \widetilde{y} B_n \widetilde{y_2} B_{n-2} \widetilde{y_3} A_{n-2} y_3 B_{n-2} \widetilde{u_2}.$$
So $B\neq B_{n-4}$. This means that, in similar way as in Case 2.1, we get $B_n,B_{n-2},B_{n-4},B\notin x_2 A_n \widetilde{x_B}$ and $A_n,B_n,B_{n-2},B_{n-4},B\notin x_B$.
Again, we have
$$|w_n|=|x_2 A_n \widetilde{x_B}|+|B|+|x_B|+|A_n \widetilde{y} B_n y A_n \widetilde{y} B_n|+|z_A|+|A_{n-2}|+|\widetilde{z_A} B_n z_2|$$
$$\leq r(n-4)+1+r(n-5)+[3r(n-2)+4]+r(n-3)+1+r(n-2)$$
$$= 4r(n-2)+r(n-3)+r(n-4)+r(n-5)+5$$
$$ < 4r(n-2)+r(n-3)+r(n-4)+r(n-5)+5+r(n-5)\leq 5r(n-2)+2.$$

Case 3.3.2.2.3) $y_3 = u_1 A_{n-4} u_3 B_{n-4} \widetilde{u_3} A_{n-4} u_3 B_{n-4} u_2$, where the rightmost $A_{n-4}$ and the leftmost $B_{n-4}$ are the right and left special letters of $y_3$, respectively.

We divide this case into two subcases: $B_{n-4}\notin y_1$ and $B_{n-4}\in y_1$.

Case 3.3.2.2.3.1) $B_{n-4}\notin y_1$.

Now $B\neq B_{n-4}$, since otherwise we would have a square in $w_n$:
$$A_{n-4} u_3 B_{n-4} \widetilde{u_3} A_{n-4} \widetilde{u_1} A_{n-2} \widetilde{y_1} A_n \widetilde{y} B_n \widetilde{y_2} B_{n-2} \widetilde{y_3} A_{n-2} y_3 B_{n-2} \widetilde{u_2} B_{n-4} \widetilde{u_3}$$
$$A_{n-4} u_3 B_{n-4} \widetilde{u_3} A_{n-4} \widetilde{u_1} A_{n-2} \widetilde{y_1} A_n \widetilde{y} B_n \widetilde{y_2} B_{n-2} \widetilde{y_3} A_{n-2} y_3 B_{n-2} \widetilde{u_2} B_{n-4} \widetilde{u_3}.$$
Similar to Case 2.1, we get $B_n,B_{n-2},B_{n-4},B\notin x_2 A_n \widetilde{x_B}$ and $A_n,B_n,B_{n-2},B_{n-4},B\notin x_B$.
Again, we have
$$|w_n|=|x_2 A_n \widetilde{x_B}|+|B|+|x_B|+|A_n \widetilde{y} B_n y A_n \widetilde{y} B_n|+|z_A|+|A_{n-2}|+|\widetilde{z_A} B_n z_2|$$
$$\leq r(n-4)+1+r(n-5)+[3r(n-2)+4]+r(n-3)+1+r(n-2)$$
$$= 4r(n-2)+r(n-3)+r(n-4)+r(n-5)+5$$
$$ < 4r(n-2)+r(n-3)+r(n-4)+r(n-5)+5+r(n-5)\leq 5r(n-2)+2.$$

Case 3.3.2.2.3.2) $B_{n-4}\in y_1$.

Now $y_1=y'_1 B_{n-4} \widetilde{u_3} A_{n-4} \widetilde{u_1}$, where the $B_{n-4}$ is the rightmost occurrence of $B_{n-4}$ in $y_1$,
and $y_2= \widetilde{u_2} B_{n-4} \widetilde{u_3} A_{n-4} y'_2$, where the $A_{n-4}$ is the leftmost occurrence of $A_{n-4}$ in $y_2$.
Remember that we really have $A_{n-4}\in y_2$, since this is a subcase of Case 3.3.2.2.

If $A_{n-2}\in y_1$ then we can take the rightmost occurrence of $A_{n-2}$ in $y'_1$ and get that $y_1=y''_1 A_{n-2} u_1 A_{n-4} u_3 B_{n-4} \widetilde{u_3} A_{n-4} \widetilde{u_1}$,
which creates a square in $\widetilde{y}$:
$$A_{n-4} u_3 B_{n-4} \widetilde{u_3} A_{n-4} \widetilde{u_1} A_{n-2} y_3 B_{n-2} \widetilde{u_2} B_{n-4} \widetilde{u_3}
  A_{n-4} u_3 B_{n-4} \widetilde{u_3} A_{n-4} \widetilde{u_1} A_{n-2} y_3 B_{n-2} \widetilde{u_2} B_{n-4} \widetilde{u_3}.$$
This means $A_{n-2}\notin y_1$.

Now we divide this case into two subcases: $B\neq A_{n-2}$ and $B=A_{n-2}$.

Case 3.3.2.2.3.2.1) $B\neq A_{n-2}$.

Now, in similar way as in Case 2.1, we get that $A_{n-2},B_n,B_{n-2},B\notin x_2 A_n \widetilde{x_B}$ and $A_n,A_{n-2},B_n,B_{n-2},B\notin x_B$.
Again, we have
$$|w_n|=|x_2 A_n \widetilde{x_B}|+|B|+|x_B|+|A_n \widetilde{y} B_n y A_n \widetilde{y} B_n|+|z_A|+|A_{n-2}|+|\widetilde{z_A} B_n z_2|$$
$$\leq r(n-4)+1+r(n-5)+[3r(n-2)+4]+r(n-3)+1+r(n-2)$$
$$= 4r(n-2)+r(n-3)+r(n-4)+r(n-5)+5$$
$$ < 4r(n-2)+r(n-3)+r(n-4)+r(n-5)+5+r(n-5)\leq 5r(n-2)+2.$$

Case 3.3.2.2.3.2.2) $B=A_{n-2}$.

Now we have that $x_B=\widetilde{y_1}=u_1 A_{n-4} u_3 B_{n-4} \widetilde{y'_1}$.
We will first show that $A_{n-4}\notin y'_1,x_2$ and $B_{n-4}\notin y'_2,z_2$.

If $A_{n-4}\in y'_1$ then we have $y_1=y''_1 A_{n-4} u_3 B_{n-4} \widetilde{u_3} A_{n-4} \widetilde{u_1}$. This creates a square in $w_n$:
$$u_3 B_{n-4} \widetilde{u_3} A_{n-4} \widetilde{u_1} A_{n-2} \widetilde{y_1} A_n \widetilde{y} B_n \widetilde{y_2} B_{n-2} \widetilde{y_3} A_{n-2} y_3 B_{n-2} \widetilde{u_2} B_{n-4} \widetilde{u_3} A_{n-4}$$
$$u_3 B_{n-4} \widetilde{u_3} A_{n-4} \widetilde{u_1} A_{n-2} \widetilde{y_1} A_n \widetilde{y} B_n \widetilde{y_2} B_{n-2} \widetilde{y_3} A_{n-2} y_3 B_{n-2} \widetilde{u_2} B_{n-4} \widetilde{u_3} A_{n-4}.$$
So $A_{n-4}\notin y'_1$. If $B_{n-4}\in y'_2$ then we have that $y_2=\widetilde{u_2} B_{n-4} \widetilde{u_3} A_{n-4} u_3 B_{n-4} y''_2$. Also this creates a square in $w_n$:
$$B_{n-4} \widetilde{u_3} A_{n-4} \widetilde{u_1} A_{n-2} \widetilde{y_1} A_n \widetilde{y} B_n \widetilde{y_2} B_{n-2} \widetilde{y_3} A_{n-2} y_3 B_{n-2} \widetilde{u_2} B_{n-4} \widetilde{u_3} A_{n-4} u_3$$
$$B_{n-4} \widetilde{u_3} A_{n-4} \widetilde{u_1} A_{n-2} \widetilde{y_1} A_n \widetilde{y} B_n \widetilde{y_2} B_{n-2} \widetilde{y_3} A_{n-2} y_3 B_{n-2} \widetilde{u_2} B_{n-4} \widetilde{u_3} A_{n-4} u_3.$$
So $B_{n-4}\notin y'_2$.
If $A_{n-4}\in x_2$ then we could take the rightmost occurrence of $A_{n-4}$ in $x_2$ and get a square in $w_n$:
$$A_{n-4} u_3 B_{n-4} \widetilde{y'_1} A_n y'_1 B_{n-4} \widetilde{u_3} A_{n-4} \widetilde{u_1} A_{n-2} u_1
  A_{n-4} u_3 B_{n-4} \widetilde{y'_1} A_n y'_1 B_{n-4} \widetilde{u_3} A_{n-4} \widetilde{u_1} A_{n-2} u_1.$$
So $A_{n-4}\notin x_2$. If $B_{n-4}\in z_2$ then we could take the leftmost occurrence of $B_{n-4}$ in $z_2$ and get a square in $w_n$:
$$u_2 B_{n-2} \widetilde{y_3} A_{n-2} y_3 B_{n-2} \widetilde{u_2} B_{n-4} \widetilde{u_3} A_{n-4} y'_2 B_n \widetilde{y'_2} A_{n-4} u_3 B_{n-4}$$
$$u_2 B_{n-2} \widetilde{y_3} A_{n-2} y_3 B_{n-2} \widetilde{u_2} B_{n-4} \widetilde{u_3} A_{n-4} y'_2 B_n \widetilde{y'_2} A_{n-4} u_3 B_{n-4}.$$
So $B_{n-4}\notin z_2$. Now we know that $A_{n-4}\notin x_2 A_n y'_1 B_{n-4} \widetilde{u_3}$ and $B_{n-4}\notin \widetilde{u_3} A_{n-4} y'_2 B_n z_2$.

Similar to Case 2.1, we get $A_{n-2},B_n,B_{n-2}\notin x_2 A_n y'_1 B_{n-4} \widetilde{u_3}$
and $A_n,A_{n-2},B_n,B_{n-2}\notin u_3 B_{n-4} \widetilde{y'_1}$ and $A_n,A_{n-2}\notin \widetilde{u_3} A_{n-4} y'_2 B_n z_2$.
From Proposition \ref{pro1} we get that $|u_1|,|u_2|\leq r(n-6)+r(n-7)+1$, where $r(n-7)=0$ if $n=7$.
Since $A_n,B_n\notin y$ and $A_{n-2}$ is the right special letter of $\widetilde{y}$, we trivially have $A_n,A_{n-2},B_n\notin \widetilde{y_2} B_{n-2} \widetilde{y_3}$.
From Lemma \ref{l3} we also get easily that $A_n,A_{n-2},B_n,B_{n-2}\notin y_3$.
Together we finally have
$$|w_n|=|x_2 A_n y'_1 B_{n-4} \widetilde{u_3}|+|A_{n-4}\widetilde{u_1}A_{n-2}u_1 A_{n-4}|+|u_3 B_{n-4} \widetilde{y'_1}|+|A_n \widetilde{y} B_n y A_n \widetilde{y} B_n|$$
$$+|\widetilde{y_2} B_{n-2} \widetilde{y_3}|+|A_{n-2}|+|y_3|+|B_{n-2}|+|\widetilde{u_2}|+|B_{n-4}|+|\widetilde{u_3} A_{n-4} y'_2 B_n z_2|$$
$$\leq r(n-4)+[2r(n-6)+2r(n-7)+5]+r(n-5)+[3r(n-2)+4]+r(n-3)+1+r(n-4)+1+$$
$$[r(n-6)+r(n-7)+1]+1+r(n-3)=3r(n-2)+2r(n-3)+2r(n-4)+r(n-5)+3r(n-6)+3r(n-7)+13$$
$$<3r(n-2)+2r(n-3)+2r(n-4)+r(n-5)+3r(n-6)+3r(n-7)+13+r(n-6)+r(n-7)\leq 5r(n-2)+2.$$
\end{proof}

As we can see, improving our upper bound was very exhausting.
If we would like to achieve Conjecture \ref{con}, we would need to use a slightly different approach.

Let us still estimate our upper bound in a closed form. Suppose first $n\geq7$ is even:
$$ r(n)\leq 5r(n-2)+4 \leq 5(5r(n-4)+4)+4\leq\ldots\leq  5^{(n-6)/2}r(6)+4(5^{(n-8)/2}+\ldots+5+1)$$
$$ < 5^{(n-6)/2}\cdot(5^3-58)+(5^{(n-8)/2+1}+\ldots+5)=5^{n/2}-58\cdot 5^{(n-6)/2}+(5^{(n-8)/2+1}+\ldots+5)<5^{n/2}<2,237^n.$$
Suppose then that $n\geq7$ is odd:
$$ r(n)\leq 5r(n-2)+4 \leq 5(5r(n-4)+4)+4\leq\ldots\leq  5^{(n-5)/2}r(5)+4(5^{(n-7)/2}+\ldots+5+1)$$
$$ < 5^{(n-5)/2}\cdot(5^{2,5}-22)+(5^{(n-7)/2+1}+\ldots+5)=5^{n/2}-22\cdot 5^{(n-5)/2}+(5^{(n-7)/2+1}+\ldots+5)<5^{n/2}<2,237^n.$$
Together with the lower bound, we finally get that $2,008^n<r(n)<2,237^n$, for $n\geq5$.

%%%%%%%%%%%%%%%%%%%%%%%%%%%%%%%%%%%%%%%%%%%
%%%%%%%%%%%%%%%%%%%%%%%%%%%%%%%%%%%%%%%%%%%
%%%%%%%%%%%%%%%%%%%%%%%%%%%%%%%%%%%%%%%%%%%

\section*{Acknowledgements}

I want to thank \v S. Starosta for making me aware of this problem,
L. Zamboni and T. Harju for insightful comments, and
J. Peltom\"{a}ki for calculating the exact value of $r(7)$ and adding all the values of $r(n)$ up to $n=7$ to the OEIS database https://oeis.org/A269560.

\end{document}